\documentclass[12pt,a4paper]{amsart}
\usepackage{amssymb}
\usepackage{hyperref}

\usepackage{graphicx,amssymb,amsfonts,epsfig,amsthm,a4,amsmath,url}
\usepackage[latin1]{inputenc}

\usepackage{amsmath, amsthm, amssymb, verbatim}
\usepackage{graphicx,amssymb,amsfonts,epsfig,amsthm,a4,amsmath,url,graphicx,enumerate}
\usepackage[latin1]{inputenc}
\usepackage{tikz}
\usepackage{textcomp}
\usetikzlibrary{matrix}
\usepackage{tikz-cd}

\usepackage{tikz}

\usepackage{lipsum}

\usepackage{mathtools}

\newtheorem{thm}{Theorem}[section]
\newtheorem{cor}[thm]{Corollary}
\newtheorem{lem}[thm]{Lemma}

\newtheorem{prop}[thm]{Proposition}
\theoremstyle{definition}

\theoremstyle{remark}
\newtheorem{rem}[thm]{Remark}

\numberwithin{equation}{section}

\pgfmathsetmacro{\xdeg}{30}
\pgfmathsetmacro{\xx}{cos(\xdeg)}
\pgfmathsetmacro{\xy}{sin(\xdeg)}

\pgfmathsetmacro{\ydeg}{120}
\pgfmathsetmacro{\yx}{cos(\ydeg)}
\pgfmathsetmacro{\yy}{sin(\ydeg)}

\pgfmathsetmacro{\zdeg}{80}
\pgfmathsetmacro{\zx}{cos(\zdeg)}
\pgfmathsetmacro{\zy}{sin(\zdeg)}

\begin{document}
\title[Warped cones and measured coupling of groups]{From the coarse geometry of warped cones to the measured coupling of groups}
\author[Das]{Kajal Das}
\address{Indian Statistical Institute, Kolkata, India}
\email{kdas.math@gmail.com}

\maketitle
\textbf{Abstract:} In this article, we prove that if two warped cones corresponding to two finitely generated groups with free, isometric, measure-preserving, actions on two compact metric spaces with probability measures are level-wise quasi-isometric (with some extra natural assumptions), then the corresponding groups are uniformly measured equivalent (UME).  It was earlier known from the works of de Laat-Vigolo and Sawicki that if two such warped cones are level-wise quasi-isometric, then their stable products are quasi-isometric. We strengthen this result and go further to prove UME of the groups. We also discuss many applications of our main result. We give countably infinite examples of groups and associated Warped cones such that the groups are mutually quasi-isometric, but the Warped cones are not mutually quasi-isometric in the sense of our main theorem. We also provide examples of two Warped cones ( which are quasi-isometric  to two different expander families) such that one of them does not quasi-isometrically embed into the other one in the sense of our main theorem.

\hspace{2mm}

 
\hspace{2mm}

\textbf{Key terms:}: Warped cones, Box spaces, Quasi-isometry, Measured equivalence, Uniformly measured equivalent, Expander graphs.

\section{Introduction}

Warped cones are geometric objects introduced by J. Roe and associated to a free action of a finitely generated group on a compact metric space \cite{Ro05}. This geometric object encodes the geometry of the Cayley graph of the group, the geometry of the manifold and the dynamics of the group (see Subsection \ref{wc-definition} for the definition of a Warped cone). This geometric object appears in the context of Coarse Baum-Connes conjecture and expander graphs. Building on the works of G. Yu, it can be shown that the warped cones associated with an amenable group provide examples of metric spaces which satisfy Coarse Baum-Connes conjecture (\cite{Yu00} ). On the other hand, using the works of Drinfeld-Margulis-Sullivan, we can provide examples of warped cones associated with Property (T) groups whose discretizations at each of the level sets  give examples of expander graphs (\cite{Mar80},\cite{LV19}). Moreover, there is an equivalence between analytic properties of the groups and geometric properties of the warped cones. The details of these results can be found in Subsection \ref{wc-example}.

However, there is a parallel connection of warped cones with Box spaces. A \textit{box space} is a geometric object associated with a finitely generated residually finite group. Let $(G,S)$ be a finitely generated residually finite group with a symmetric generating set $S$ and  $\{G_n\}_{n\in\mathbb{N}}$ be a decreasing sequence of finite index normal subgroups of $G$ whose intersection is trivial. We consider the Cayley graph structure on $G$  arising from the generating set $S$. We push-down this Cayley graph structure on $(G/G_n,\bar{S}_n)$,  where 
 $\bar{S}_n$ is the image of $S$ in $G/G_n$ under the quotient map.  The box space of $G$ w.r.t. $\{G_n\}_{n\in\mathbb{N}}$ is defined as the disjoint union of $\sqcup_{n\in\mathbb{N}}G/G_n$ and denoted by $\square_{G_n}G$. We can give a metric on a box space as follows: We consider the Cayley graph metric on each of $G/G_n$ and we assign a metric on the union such that the distance between two distinct copies $G/G_n$ and $G/G_{n+k}$  tends to infinity as $n\to \infty$. The $t$-th level of a Warped cone can be compared to the $n$-th level of a box space. It has been observed in Khukhro-Valette that if two box spaces are quasi-isometric, the the groups are quasi-isometric (\cite{KV17} ). Later, the author has generalized this result to uniform measured equivalence: if two box spaces  are quasi-isometric, then the groups are uniformly measured equivalent (\cite{Da18} ).  Motivated by the result of Khukhro-Valette, de Laat-Vigolo ( see \cite{LV19}) proves the following theorem for warped cones (independently obtained by D. Sawicki in \cite{Sa19}) :

\begin{thm}\label{coarseequi1}
Let $\Gamma$ and $\Lambda$ act freely and isometrically on compact Riemannian manifolds $(M, d^M)$ and $(N, d^N)$ , respectively.  Suppose $\mathcal{O}_\Gamma M$ and $\mathcal{O}_\Lambda N$ are the Warped cones associated to $(M, d_M)$ and $(N, d_N)$, respectively. Moreover, we assume that there exists a family $\{\Phi_t:(M_t\rightarrow N_t\}_{t\geq 0}$ of $(K, C)$-QI  between  $\mathcal{O}_\Gamma M$ and $\mathcal{O}_\Lambda N$, respectively, for some $K\geq 1$ and $C\geq 0$. Then, $\Gamma\times\mathbb{Z}^m$ and $\Lambda \times \mathbb{Z}^n$ are $(K, C)$-QI, where $m= dim M $ and $n= dim N$. 
\end{thm}

Motivated by the above-mentioned  result of the author on Box space \cite{Da18}, P. Nowak and D. Sawicki asked the author whether we can prove an analogous result of Theorem \ref{coarseequi1} in the setting of uniform measured equivalent.  In this article, we affirmatively answer this question. 

\begin{thm}(Main Theorem)\label{mainthm}
Let $\Gamma$ and $\Lambda$ be two finitely generated  groups which are acting on two compact metric spaces with probability measures $(X, d^X, \mu^X)$ and $(Y, d^Y, \mu^Y)$, respectively,  by free,  isometric and probability measure preserving ergodic actions. Suppose, there exist a family $\{\Phi_t: (X_t, d^X_t)\rightarrow (Y^t, d^Y_t)\}_{t\geq 1}$ of Borel maps and constants $K\geq 1$, $C\geq 0$ satisfying the following properties for all $t$: 
\begin{itemize}
\item[1.] $\Phi_t$ is a $(K,C)$-quasi-isometric embedding, i.e., 
$$ \frac{d^X_t(x,x')}{K}-C\leq d^Y_t\big(\Phi_t(x),\Phi_t(x')\big)\leq K d^X_t(x,x')+C ;$$
\item[2.] $\Phi^t$ maps the $\Gamma$-orbits of $X_t$ to the $\Lambda$-orbits of $Y_t$;
\item[3.] $\Phi^t$ is $C$-onto along the orbits, i.e., for all $x\in X_t$ and $\lambda\in\Lambda$ there exists $\gamma\in \Gamma$ such that $d^Y_t\big(\Phi_t(\gamma x),\lambda \Phi_t(x)\big)\leq C$;
\item[4.] $\Phi^t$ is $K$-measure preserving, i.e., for all Borel sets $A$ of $Y_t$
$$\frac{\mu^Y(A)}{K}\leq \mu^X\big((\Phi_t)^{-1}(A)\big)\leq K \mu^Y(A).$$
\end{itemize}

Then,
\begin{itemize}
\item[(A)] the groups $\Gamma$ and $\Lambda$ are quasi-isometric;
\item[(B)] the groups $\Gamma$ and $\Lambda$ are uniformly measured equivalent. 
\end{itemize}

\end{thm}

We remark that all four assumptions on the family of quasi-isometries $\Phi^t$ is very natural in the category of warped cones. In case of `box spaces', a $(K,C)$-quasi-isometry map between $n$th level of two box spaces is automatically Borel. Since a group acts transitively on its quotient groups,  a quasi-isometry map between $n$th level of two box spaces preserves the orbits and it is $C$-onto along the orbits. Finally, any $(K,C)$-quasi-isometry satisfies the fourth assumption of our main theorem.

\subsection{Acknowledgements}

I would like to thank Prof. Piotr Nowak and Dr. Damian Sawicki for asking me this question (Theorem \ref{mainthm}). I would like to thank Prof. Romain Tessera for numerous helpful suggestions. This project has received funding from the European Research Council (ERC) under the European Union's Horizon 2020 research and innovation programme (grant agreement no. 677120-INDEX) ( provided by Prof. Piotr Nowak) and from National Board for Higher Mathematics, India.

\subsection{Organization}
In Section 2, we introduce our necessary definitions, notations and abbreviations. In Section 3, we construct quasi-isometries between the groups. In Section 4, we prove our main theorem, i.e., we construct a measured coupling space between the groups. In Section 5 we discuss many applications of our main theorem.

\subsection{Warped Cones}\label{wc-definition}

Let $X$ be a compact metric space with metric $d$, and for every $t \geq 1$. Consider the Euclidean cone, Cone($X$)  over $X$, which is defined as the union of $\{X_t\}_{t\geq 1}$, where each $t$-th level $X_t$ is the metric space $(X,d)$. 
Given an action : $\Gamma\curvearrowright (X,d)$ by homeomorphisms, we define the $t$-level of the Warped cone as the metric space $(X, d^X_t )$ where $d^X_t$ is the warped distance, i.e. the largest metric satisfying

\begin{itemize}
\item[1.] $d^X_t(x,y)\leq t\hspace{1mm}d(x,y)$ for all $x, y\in X$, 
\item[2.] $d^X_t(x,\gamma\cdot x)\leq |\gamma|$  for all $x\in X$ and $\gamma\in\Gamma$,  
\end{itemize}
where $|\gamma|$ denotes the word-metric with respect to a generating subset $S$ of $\Gamma$. 
Note that the definition depends on the choice of the generating set $S$. However, the coarse structure induced by the warped metric does not depend on the generating sets. We denote the Warped cone by $\mathcal{O}_\Gamma X$ and $t$-th level of the Warped cone by $(X_t, d^X_t)$.  The Warped distance between two points $x,y\in (X_t, d^X_t)$ can be expressed as follows:

$$d^X_t(x,y)=\mbox{inf}_{\gamma\in\Gamma}\hspace{1mm}[t\hspace{1mm}d(x,\gamma y) + |\gamma|].$$


\begin{rem}
 Since Coarse Baum-Connes  conjecture  is concerned about bounded geometry metric spaces and the isometric action of groups give rise to bounded geometry warped cones  (see Proposition 1.10, \cite{Ro05} ), in this article we mostly consider  warped cones with isometric  group action. 
 \end{rem}

\subsection{Some interesting examples of warped cones}\label{wc-example}

\begin{itemize}
\item[1.]   Let $Y$ be a compact manifold (or finite simplicial complex)
and let $\Gamma$ be an amenable group acting by Lipschitz homeomorphisms on $Y$ .
Then the warped cone $\mathcal{O}_\Gamma(Y )$ has property A of G. Yu (Cor. 3.2 \cite{Ro05}) . 
Conversely, suppose that the warped cone $\mathcal{O}_\Gamma(G)$,
has property A, where $\Gamma$ is a dense subgroup of compact Lie group $G$.  Then $\Gamma$ is amenable (\cite{Ro05}).
 For a generalized version of this result,  see \cite{SW21}. 

\item[2.] Suppose that the warped cone $\mathcal{O}_\Gamma(G)$, as defined above, is
uniformly embeddable in Hilbert space. Then $\Gamma$ has the Haagerup property .
In particular, if $\Gamma$ has property T, then $\mathcal{O}_\Gamma(G)$ cannot be uniformly embedded
in Hilbert space (\cite{Ro05}).  For a generalized version of this result,  see \cite{SW21}. 

\item[3.]   Some expander graphs can be constructed using warped cone. Sullivan and Margulis prove that $SO(n, \mathbb{Z}[1/5])$ ($n > 4$)  is a Property (T) group embedded densely inside the compact Lie group $SO(n)$. We consider the warped cone associated with the natural action of $SO(n, \mathbb{Z}[1/5])$ on $SO(n)$ ( with a bi-invariant metric and the Haar measure). If we take graph-approximations of each $t$-th level of the Warped cone, this sequence of graphs will be an expander sequence, more particularly, the associated warped cone will be uniformly quasi-isometric to an expander sequence (see  \cite{LV19} for more details). On the other hand, 
if $a,b\in SO(3,\bar{\mathbb{Q}})$ are two matrices with algebraic coefficients and generate a non-abelian free group $F_2$, then their action by rotations on $\mathbb{S}^2$ has spectral gap (due to Bourgain and Gamburd). It follows from \cite{Vi19} that the level sets of the Warped cone $\mathcal{O}_{{F_2}}(\mathbb{S}^2)$ are uniformly quasi-isometric to an expander family. 

\end{itemize}

\section{Preliminaries: some definitions, notations and abreviations:}

\subsection{Quasi-isometry}\label{qi} Let $X$ and $Y$ be two metric spaces. A map $f: X\rightarrow Y$ is said 
to be a $(K,C)$-quasi-isometry, where $K\geq 1, C\geq 0$,  if the following conditions are satisfied: 

\begin{itemize}
\item $\frac{d_X(x_1,x_2)}{K}-C\leq d_Y\big(f(x_1),f(x_2)\big)\leq K d_X(x_1,x_2)+C$ for all $x_1, x_2\in X$ ;
\item the image $f(X)$ is $C$-dense in $Y$, i.e., for any $y\in Y$ there exists a $x\in X$ such that $y$ is inside the $C$-radius ball of $f(x)$. 
\end{itemize}

\subsection{Measured Equivalence and Uniform Measured Equivalence}\label{me} In \cite{Gr93} (p. 6), Gromov first formulates a topological criterion for quasi-isometry and introduces measured equivalence as a measure theoretic counterpart of quasi-isometry. Two countable discrete groups $\Gamma$ and $\Lambda$ are called \textit{Measured Equivalent}(ME) if they have commuting measure preserving free actions on a Borel space $(X, \mu)$ with finite measure Borel fundamental domains, say $X_\Gamma$ and $X_\Lambda$, respectively. The space $(X, \mu)$ is called a `measured coupling space'  for the groups $\Gamma$ and $\Lambda$.

\textit{Uniform Measured Equivalence}(UME) is a sub-equivalence relation of `Measure Equivalence' on finitely generated groups introduced by Shalom in \cite{Sh04}.  If, in a measured equivalence relation, the action of an element of one group, say $\Gamma$, on the fundamental domain of another group, say $X_\Lambda$, is covered by finitely many $\Lambda$-translates of  $X_\Lambda$, then these two groups are called UME.

 \bigskip

\section{ QI of the groups and Proof of Part (a) of our main theorem}

Suppose we are in the set up of Theorem \ref{mainthm}. We define $A^X_t(R)=\{x\in X | d^X_t(x,\gamma x)= |\gamma| \forall \gamma\in B^\Gamma_1(R) \}$ and $A^Y_t(R)=\{y\in Y | d^Y_t(y,\lambda y)= |\lambda| \forall \lambda\in B^\Lambda_1(R) \}$for $R\geq 0$ and $t\geq 1$. 

\begin{lem}\label{increasingsets}
$A^X_{t_1}(R)\subseteq A^X_{t_2}(R)$  and $A^Y_{t_1}(R) \subseteq A^Y_{t_2}(R)$  if $t_1\leq t_2$ for all $R\geq 0$. 
\end{lem}
\begin{proof}
  Since $d^X_{t_1}(x,\gamma x)\leq d^X_{t_2}(x,\gamma x)$ for $t_1\leq t_2$ , we have $A^X_{t_1}(R)\subseteq A_{t_2}(R)$.  Similarly, we can prove that $A^Y_{t_1}(R) \subseteq A^Y_{t_2}(R)$. 
\end{proof}

\begin{lem}\label{unioniswhole}
Fix $R\geq 0$,  $x\in X$ and $y\in Y$. Then, there exists $t(R,x)\geq 1$ such that $d^X_t(x,\gamma x)=|\gamma |$  for all $t>t(R,x)$ and for all $\gamma\in B^\Gamma_1(R)$, and there exists $t(R,y)$ such that $d^Y_t(x,\lambda x)=|\lambda |$  for all $t> t(R,y)$ and for all $\lambda\in B^\Lambda_1(R)$. In particular, $\cup_{t\geq1} A^X_t(R)= X$ and $\cup_{t\geq1} A^Y_t(R)= Y$ for all $R\geq 0$. 
\end{lem}
\begin{proof}

We fix $R\geq 0$, $x\in X$, and  $\gamma\in B^\Gamma_1(R)$. By definition, 

$$d^X_t(x,\gamma x)=\mbox{inf}\{td^X(x,\gamma' \gamma x)+|\gamma'| : \gamma'\in \Gamma\}.$$

Since the infimum in the above definition of $d^X_t(x,\gamma x)$ is a minimum, we obtain that there exists  $\gamma_t \in \Gamma$ such that $d^X_t(x,\gamma x)=td^X(x,\gamma_t \gamma x)+|\gamma_t| $, which implies that $|\gamma_t| \leq |\gamma |$.  If $\gamma_t\neq \gamma^{-1}$, we obtain that  $d^X(x,\gamma_t \gamma x)\geq \epsilon >0$, where 
$$\epsilon=\mbox{inf}\{d^X(x,\gamma' \gamma x): \gamma'\neq \gamma^{-1}, |\gamma'| \leq R\}.$$
It implies that $t \leq \frac{|\gamma |}{\epsilon}\leq \frac{R}{\epsilon}$. Therefore, $\gamma_t=\gamma^{-1}$ for $t> t(R,x)$ where $t(R,x)=\frac{R}{\epsilon}$.

Similarly, we can prove the result for $Y$. 

\end{proof}


We define 

$$D_t(R)=[\cap_{\gamma\in B_1^\Gamma(R)}\gamma A^X_t(R)]\cap[\cap_{\lambda\in B_1^\Lambda} \Phi_t^{-1}(\lambda A^Y_t(KR+C))]$$

For all $x\in X$ and $t\geq 1$, we define $\alpha_t(\cdot,x):\Gamma\rightarrow\Lambda$ as a map satisfying 
the following equality:

$$\Phi_t(\gamma x)=\alpha_t(\gamma,x) \Phi_t(x)$$ 

for all $t\geq 1$, $\gamma\in\Gamma$ and $x\in X$. 

\begin{lem}\label{alphatcocycle}
For all $t\geq 0$, the map $\alpha_t:\Gamma\times X\rightarrow \Lambda$ is a cocycle. 
\end{lem}
\begin{proof} Let $\gamma_1,\gamma_2\in\Gamma$ and $x\in X$. 
From the definition of $\alpha_t$, we obtain that $\Phi_t(\gamma_1\gamma_2 x)=\alpha_t(\gamma_1\gamma_2,x) \Phi_t(x)$. On the other hand, 
$$\Phi_t\big(\gamma_1(\gamma_2 x)\big)=\alpha_t(\gamma_1,\gamma_2 x)\Phi_t(\gamma_2 x)=\alpha_t(\gamma_1,\gamma_2 x)\alpha_t(\gamma_2,x)\Phi_t(x).$$ Since $\Lambda$-action on $Y$ is free, we obtain that $\alpha_t(\gamma_1\gamma_2,x)=\alpha_t(\gamma_1,\gamma_2 x)\alpha_t(\gamma_2,x)$. 
\end{proof}

\begin{lem}\label{alphatlocqi}
For $x\in D_t(R)$, the cocyle $\alpha_t( \cdot, x)|_{B_1^\Gamma(R)}$ defines a $(K,C)$-Quasi-isometry from $B_1^\Gamma(R)$ to $B_1^\Gamma(KR+C)$. 
\end{lem}
\begin{proof}
This lemma follows from the definition of $D_t(R)$. 
\end{proof}

\begin{lem}\label{measdtrtend1}
For all $R, N\in \mathbb{N}$, there exists $t(R,N)\geq 0$  such that $\mu^X(D_t(R))\geq (1-\frac{1}{N})$ for all $t\geq t(R,N)$. 
\end{lem}
\begin{proof} We fix $R\geq 0$. From Lemma \ref{increasingsets} and Lemma \ref{unioniswhole}, we obtain that $\mu^X\big(A^X_t(R)\big)$ monotonically increases and tends to 1 as $t\rightarrow\infty$. Since $\Gamma$-action preserves the measure $\mu^X$, we obtain that $\mu^X\big(\gamma A^X_t(R)\big)$ monotonically increases and tends to 1 as $t\rightarrow\infty$ for all $\gamma\in B_1^\Gamma(R)$. Similarly, $\mu^Y\big(\lambda A^Y_t(KR+C)\big)$ monotonically increases and tends to 1 as $t\rightarrow\infty$ for all $\lambda\in B_1^\Lambda(KR+C)$, which implies that $\mu^Y[Y-\big(\lambda A^Y_t(KR+C)\big)]$ monotonically decreases and tends to 0 as $t\rightarrow\infty$ for all $\lambda\in B_1^\Lambda(KR+C)$. From condition (4) of our main theorem, we obtain that 
$$\mu^X[\Phi_t^{-1}\{Y-\big(\lambda A^Y_t(KR+C)\big)\}]\leq K \mu^Y[Y-\big(\lambda A^Y_t(KR+C)\big)].$$
Therefore, $\mu^X[\Phi_t^{-1}\big(\lambda A^Y_t(KR+C)\big)]$ monotonically increases and tends to 1 as $t\rightarrow\infty$ for all $\lambda\in B_1^\Lambda(KR+C)$. Since, 

$$D_t(R)=[\cap_{\gamma\in B_1^\Gamma(R)}\gamma A^X_t(R)]\cap[\cap_{\lambda\in B_1^\Lambda} \Phi_t^{-1}(\lambda A^Y_t(KR+C))],$$

$\mu^X\big(D_t(R)\big)$ monotonically increases and tends to 1 as $t\rightarrow\infty$. Hence, we have our lemma.

\end{proof}

We define $$D=\cap_{R=0}^\infty \cup_{N\geq 0} D_{t(N,R)}(R).$$

It is easy to observe that $\mu^X(D)=1$. We fix an ultra-filter $\mathcal{U}$ over the set $T=\{t(R,N)\}_{R,N\in\mathbb{N}_0}$. Define $\alpha(\cdot,x)=lim_{t\rightarrow \mathcal{U}} \alpha_t(\cdot,x):\Gamma\rightarrow\Lambda$. for all $x\in D$. 

\begin{prop}\label{alphaqi}
$\alpha(\cdot,x):\Gamma\rightarrow\Lambda$ is a $(K,C)$-Quasi-isometry for all $x\in D$. In particular, it proves part (A) of our main theorem.
\end{prop}
\begin{proof}
We fix $\gamma_1$ and $\gamma_2$ in $\Gamma$ and $x\in D$. We define
$$A=\{t\in T : \alpha_t(\gamma_1,x)=\alpha(\gamma_1,x), \alpha_t(\gamma_2,x)=\alpha(\gamma_2,x)\}.$$
 Now, for all $t\in A$, 
$$ \frac{|\gamma_1\gamma_2^{-1}|}{K}-C\leq |\alpha_t(\gamma_1,x)(\alpha_t(\gamma_2,x))^{-1} |\leq K |\gamma_1\gamma_2^{-1}|+C .$$
Therefore, 
$$ \frac{|\gamma_1\gamma_2^{-1}|}{K}-C\leq |\alpha(\gamma_1,x)(\alpha(\gamma_2,x))^{-1} |\leq K |\gamma_1\gamma_2^{-1}|+C .$$
Now, we fix $\lambda_0\in\Lambda$ and $x\in D$. From condition (3) of our main theorem, we obtain that for all $t\geq 1$ there exists $\gamma_t\in\Gamma$ such that 

$$|\alpha(\gamma_t,x)\lambda_0^{-1}|\leq C .$$

Since $\{\gamma_t\}_{t\geq 1}$ is a bounded set in $\Gamma$, there exists, a subset $B$ of $T$ in the ultrafilter $\mathcal{U}$ and an element $\gamma_0\in\Gamma$ such that $\gamma_t=\gamma_0$ for all $t\in B$. Therefore, 
$$|\alpha(\gamma_0,x)\lambda_0^{-1}|\leq C .$$ 
Hence, $\alpha(\cdot,x):\Gamma\rightarrow\Lambda$ is a $(K,C)$-Quasi-isometry for all $x\in D$.

\end{proof}

Let $X'=\cap_{\gamma\in\Gamma} \gamma D$. We observe that $X'$ is a Borel subset of $X$ closed under $\Gamma$-action with measure 1. We end this subsection by proving that $\alpha(\cdot,x)$ is not only a quasi-isometry, the restriction of $\alpha$ on $\Gamma\times X'$ is a cocycle. This will be required in the next subsection.

\begin{prop}\label{alphacocycle}
$\alpha: \Gamma\times X' \rightarrow \Lambda $ is a cocyle.
\end{prop}
\begin{proof}
We need to prove $\alpha(\gamma_1\gamma_2, x)=\alpha(\gamma_1,\gamma_2\cdot x)\alpha(\gamma_2,x)$ for all $\gamma_1,\gamma_2\in \Gamma$ and $x\in X$.  We fix $\gamma_1,\gamma_2\in \Gamma$ and $x\in X$. We observe that $$A=\{t\in T : \alpha_t(\gamma_1\gamma_2, x)=\alpha(\gamma_1\gamma_2,x)\hspace{1mm}\mbox{and}\hspace{1mm} \alpha_t(\gamma_2, x)=\alpha(\gamma_2,x) \}$$

belongs to the ultrafilter $\mathcal{U}$. 

Now, we have 
$$\alpha(\gamma_1\gamma_2,x) \big(\alpha(\gamma_2,x)\big)^{-1}=\alpha_t(\gamma_1\gamma_2,x) \big(\alpha_t(\gamma_2,x)\big)^{-1}=\alpha_t(\gamma,\gamma_2\cdot x)$$
 for all $t\in A$. Therefore, $\alpha(\gamma_1\gamma_2,x) \big(\alpha(\gamma_2,x)\big)^{-1}=\alpha(\gamma_1,\gamma_2\cdot x)$, which is equivalent to saying that 
 $$\alpha(\gamma_1\gamma_2, x)=\alpha(\gamma_1,\gamma_2\cdot x)\alpha(\gamma_2,x). $$

\end{proof}

\section{Construction of measured coupling and proof of part (b) of our main theorem}

Let $Z=X'\times\Lambda$. We define $\Gamma$-action and $\Lambda$-action on $Z$ as follows:

$$\gamma\cdot (x,\lambda')=(\gamma x,\alpha(\gamma,x)\lambda')$$

for all $\gamma\in\Gamma$, $x\in X'$ and $\lambda'\in\Lambda$, and 

$$\lambda \cdot (x,\lambda')=(x,\lambda'\lambda^{-1})$$

for all $x\in X'$ and $\lambda,\lambda'\in\Lambda$. 

We consider the restriction of the measure $\mu^X$ to $X'$ and tessellate it on whole $Z$ by $\Lambda$-action. We denote this measure on $Z$ by $\rho$. 

\vspace{5mm}

\textbf{Proof of part (B) of our main theorem:} It is easy to verify that 
\begin{itemize}
\item[(i)] $\Gamma$-action and $\Lambda$-action on $Z$ commute;
\item[(ii)] The measure $\rho$ is invariant under the action of $\Gamma$ and $\Lambda$. 
\end{itemize}
Clearly, $Z_\Lambda=X'\times \{1\}$ is a Borel fundamental domain of $\Lambda$ with measure 1. Now, we will show the existence of a Borel fundamental of $\Gamma$ with finite measure. Since $\alpha(\cdot,x):\Gamma\rightarrow \Lambda$ is $C$-onto for all $x\in X'$, there exists a finite subset $F$ of $\Lambda$ such that $\big(\alpha(\cdot,x)(\Gamma)\big)^{-1} F=\Lambda$ for all $x\in X'$. Let $K=X'\times F$. Now, we show that 
$\Gamma K=Z$. We fix $(x',\lambda)\in Z$. By the previous argument, there exists $\gamma\in\Gamma$ and $c\in F$ such that $\big(\alpha(\gamma^{-1},x')\big)^{-1} c=\lambda$. It is easy to verify that 
$$\gamma\cdot (\gamma^{-1}x',c)=(x',\lambda).$$
Therefore, $\Gamma K=Z$. It is easy to see that  $\rho(K)$ is finite. Therefore, there exists a Borel fundamental domain of $\Gamma$ inside $K$, say $Z_\Gamma$, with $\rho(Z_\Gamma)$ finite. Therefore, we obtain that $Z$ is a coupling space for $\Gamma$ and $\Lambda$, hence, $\Gamma$ and $\Lambda$ are measured equivalent.

Now, we prove that $\Gamma$ and $\Lambda$ are uniformly measured equivalent. We fix $\gamma\in\Gamma$ and consider the set $\gamma\cdot Z_\Lambda=\gamma\cdot [x'\times \{1\}]$. Since $\alpha_(\cdot,x):\Gamma\rightarrow \Lambda$ is a $(K,C)$-quasi-isometry for all $x\in X'$, $\alpha(\gamma,x)$ takes only finitely many values in $\Lambda$. Therefore, $\gamma\cdot Z_\Lambda$ can be covered by finitely many $\Lambda$-translates of $Z_\Lambda$. Now, we fix $\lambda\in\Lambda$ and consider the set $\lambda\cdot Z_\Gamma$.
For any element $(x,\lambda')\in Z_\Gamma$ , $\lambda\cdot (x,\lambda')=(x,\lambda'\lambda^{-1})$. Let $\gamma$ be the unique element such that after multiplying $(x,\lambda'\lambda^{-1})$ by $\gamma$, it gets translated back to $Z_\Gamma$.  Since 
$$\gamma\cdot\big(\lambda\cdot (x,\lambda')\big)=\big(\gamma x, \alpha(\gamma,x)\lambda'\lambda^{-1}\big),$$
 $\alpha(\gamma,x)\lambda'\lambda^{-1}\in F$, which implies that $\alpha(\gamma,x)\in F\lambda F^{-1}$. Since $F\lambda F^{-1}$ is a bounded set in $\Lambda$, there exists only finitely many $\gamma$ satisfying 
 $\alpha(\gamma,x)\in F\lambda F^{-1}$ for all $x\in X'$. Therefore, $\lambda\cdot Z_\Gamma$ can be covered by finitely many $\Gamma$-translates of $Z_\Gamma$. Hence, $\Gamma$ and $\Lambda$ are uniformly measured equivalent. \hfill\(\Box\)

\begin{rem}\label{qiembed}
 In our main theorem, if there exists a family $\{\Phi_t: (X_t, d^X_t)\rightarrow (Y^t, d^Y_t)\}_{t\geq 1}$ of Borel maps and constants $K\geq 1$, $C\geq 0$ satisfying only conditions (1), (2) and the following one:  
 for all Borel sets $A$ of $Y_t$
$$ \mu^X\big((\Phi_t)^{-1}(A)\big)\leq K \mu^Y(A),$$
 
 then 
 \begin{itemize}
 \item[(A)] there exists a $(K,C)$-quasi-isometric embedding from $\Gamma$ into $\Lambda$;
 \item[(B)] there exists a UME-embedding from $\Gamma$ into $\Lambda$. 
 \end{itemize}
 
 We can follow the same strategy taken in the proof of Theorem \ref{mainthm} to prove these two results. 

\end{rem}

\section{Applications}

\subsection{Distinguishing Warped cones up to quasi-isometry and quasi-isometric embedding:}

We are now ready to discuss the questions about quasi-isometry and quasi-isometric embedding among the Warped cones. 

\bigskip 

 \textbf{I. Groups with Property T and  groups with Haagerup Property:} Since  Property (T) and Haagerup property are invariant under measure equivalence, these properties are also invariant under quasi-isometry between Warped cones satisfying the properties in Theorem \ref{mainthm}. By Remark \ref{qiembed}, if there exists a quasi-isometric embedding of one Warped cone into another Warped cone, there exists a UME-embedding of the first group into the second group. But, an infinite group with Property (T) can not have UME-embedding into a Group with Haagerup property (for the proof, see Corollary 1.2 (i)). Therefore, we obtain the following corollary. 
 
 \begin{cor}
 There does not exist a quasi-isometric embedding from the Warped cone $\mathcal{O}_{SO(n,\mathbb{Z}[1/5])}SO(n,\mathbb{R})$ into $\mathcal{O}_{F_2}(\mathbb{S}^2)$ in the sense of Remark \ref{qiembed}.
 \end{cor}

\bigskip

\textbf{II. Lattices of $SL_n(\mathbb{R})$ ($n\geq 3$) and hyperbolic groups:} 
 
By \cite{BFGM07} (p. 4), the lattices in $SL_n(\mathbb{R})$ ($n\geq 3$) have Property $F_{L^p}$, where $1<p<\infty$, i.e., any affine isometric action of these groups on an $L^p$-space  ($1<p<\infty$) has a fixed point. On the other hand, by Theorem 1.1 of \cite{Yu05}, hyperbolic groups are a-$L^p$-menable, i.e., they admit a metrically proper affine isometric action on an $l^p$-space for some $2\leq p <\infty$. But, there does not exist a UME-embedding from an infinite group with Property $F_{L^p}$ into an a-$L^p$-menable group (for the proof, see Corollary 1.2 (i)). Therefore, we obtain the following corollary. 

\begin{cor}
There does not exist a quasi-isometric embedding from a Warped cone associated to a lattice in $SL_n(\mathbb{R})$ ($n\geq 3$) into a Warped cone associated to a lattice in  $Sp(n,1)$ ($n\geq 2$) in the sense of Remark \ref{qiembed}.
\end{cor}

\bigskip

\textbf{III. $SL_n(\mathbb{Z})$ and $SL_m(\mathbb{Z})$, where $n>m$ and $n,m\geq 3$:}
 
From Theorem 1.5 of \cite{Sh04}, if there exists a UME-embedding from a group $\Gamma$ into a group $\Lambda$, then  $cd_R \Gamma \leq cd_R \Lambda$, where  $cd_R \Gamma$ and $cd_R \Lambda$ denote the cohomological dimensions of $\Gamma$ and $\Lambda$  with coefficients in a ring $R$, respectively. 
 By a result of Borel-Serre, we get explicit values of cohomological dimension of $SL_n(\mathbb{Z})$
 with coeffcients in the ring $\mathbb{Q}$ : $cd_\mathbb{Q} SL_n(\mathbb{Z})=dim N$, where $N$ is set the upper triangular unipotent matrices which appears in the Iwasawa decomposition of $SL_n(\mathbb{R})$.
 Therefore, using Remark \ref{qiembed}, we obtain the following corollary: 
 
 \begin{cor}
 There exists no quasi-isometric embedding of a Warped cone associated with  $SL_n(\mathbb{Z})$ into a Warped cone associated with $SL_m(\mathbb{Z})$, where $n>m$, in the sense of Remark \ref{qiembed}.  
 \end{cor}

\subsection{Examples of not quasi-isometric Warped cones of quasi-isometric groups:}

In this subsection, we show a countable class of groups and associated Warped cones such that two associated Warped cones are not mutually quasi-isometric in the sense of Theorem \ref{mainthm}, but the groups are mutually quasi-isometric. We consider the following class of finitely generated residually finite groups $\Gamma_{p,q,r}:=(F_p\times F_q)\ast F_r$, where $F_p$, $F_q$ and $F_r$ are 
free groups with $p$, $q$ and $r$ generators, respectively, and $p, q, r\geq 2$. Let $X_{p,q,r}$ be the profinite completion of $\Gamma_{p,q,r}$. Let $\mathcal{O}_{\Gamma_{p,q,r}}X_{p,q,r}$ denote the Warped cone associated with the natural action of $\Gamma_{p,q,r}$ on $X_{p,q,r}$.

\begin{cor}\label{notqiwc}
There exists an infinite subclass of $\{\mathcal{O}_{\Gamma_{p,q,r}}X_{p,q,r} : (p,q,r)\in\mathbb{N}\times\mathbb{N}\times\mathbb{N} \}$ such that any two Warped cones of the subclass are not mutually quasi-isometric in the sense of Theorem \ref{mainthm}, but the corresponding groups are mutually quasi-isometric. 
\end{cor}

\begin{proof}

Since all free groups with at least two generators are commensurable, the groups $\Gamma_{p,q,r}$ for different $(p,q,r)$ are quasi-isometric. Using Properties 1.5 and Example 1.6 of \cite{Ga02} (p. 12, 13), we compute the first and second $l^2$-betti numbers of this group:
$\beta_1^{(2)}[(F_p\times F_q)\ast F_r]=r$ and  $\beta_2^{(2)}[(F_p\times F_q)\ast F_r]=(p-1)(q-1)$. From a deep of result of \cite{Ga02}, we obtain that if two groups are measured equivalent, the $l^2$-betti numbers of 
the groups must be proportional.  Therefore, we can obtain an infinite subclass of groups from 
$\{\Gamma_{p,q,r} | (p,q,r)\in\mathbb{N}\times\mathbb{N}\times\mathbb{N} \}$ which are not mutually measured equivalent. Now, using Theorem \ref{mainthm}, there exists an infinite subclass of $\{\mathcal{O}_{\Gamma_{p,q,r}}X_{p,q,r} : (p,q,r)\in\mathbb{N}\times\mathbb{N}\times\mathbb{N} \}$  such that any two Warped cones of the subclass are not mutually quasi-isometric in the sense of Theorem \ref{mainthm}, but the corresponding groups are mutually quasi-isometric. 
\end{proof}

\begin{rem}
The above corollary shows that Theorem \ref{mainthm} is stronger than Theorem \ref{coarseequi1}, i.e., there are examples of groups whose Warped cones can not be distinguished up to quasi-isometry by \ref{coarseequi1} but can be distinguished by \ref{mainthm}. 
\end{rem}


\begin{thebibliography}{KM98b}



\bibitem[1]{BFGM07} U Bader, A Furman, T Gelander, N Monod, Property (T) and rigidity for actions on
Banach spaces. Acta Mathematica, March 2007, Volume 198, Issue 1, pp 57-105.






\bibitem[2]{Da18} K. Das, From the geometry of box spaces to the geometry and measured couplings of the groups. K Das.  Journal of Topology and Analysis. Vol. 10, No. 02, pp. 401- 420 (2018). 



\bibitem[3]{Ga02} D. Gaboriau, Invariants $l^2$ de relations d'\'{e}quivalence et de groupes, Publ. Math. Inst. Hautes Etudes Sci. 95 (2002),
93-150.

\bibitem[4]{Gr93} M. Gromov, Asymptotic invariants of infinite groups. In G. Niblo and M. Roller (Eds.), Geometric group theory II, number 182 in  LMS lecture notes. Camb. Univ. Press, 1993.



\bibitem[5]{KV17} A. Khukhro and A. Valette, Expanders and box spaces. Adv.  Math. 314 (2017), 806-834. 


\bibitem[6]{LV19} T. de Laat and F. Vigolo, Superexpanders from group actions on compact manifolds, 
Geom. Dedicata 200 (2019), 287-302.

\bibitem[7]{Mar80} G.A. Margulis, Some remarks on invariant means. Monatsh. Math., 90 (1980), 233-235 .

\bibitem[8]{Mar88} G. Margulis, Explicit group theoretic constructions of combinatorial schemes and their applications in the construction of  expanders and concentrators, Problems Inform. Transmission 24, 1988.


\bibitem[9]{Ro05} J. Roe, Warped cones and property A. Geom. Topol., 9 (2005), 163-178.

\bibitem[10]{Sa19} D. Sawicki, Warped cones, (non-)rigidity, and piecewise properties. Proc. Lond. Math. Soc. (3) 118 (2019), no. 4, 753-786

\bibitem[11]{Sh04} Y. Shalom, Harmonic analysis, cohomology, and the large scale geometry of amenable groups.
Acta Mathematica 193 (2004), 119--185.

\bibitem[12]{SW21}  D. Sawicki and J. Wu, Straightening warped cones.  Journal of Topology and Analysis. Vol. 13, No. 04, pp. 933-957 (2021). 



\bibitem[13]{Vi19} F. Vigolo, Measure expanding actions, expanders and warped cones. Trans. Amer. Math. Soc. 371 (2019), no. 3, 1951-1979. 





\bibitem[14]{Yu00} G. Yu, The coarse Baum-Connes conjecture for spaces which admit a uniform embedding
into Hilbert space. Invent. Math., 139(1):201-240, 2000.

\bibitem[15]{Yu05} G. Yu, Hyperbolic groups admit proper affine isometric action on $l^p$-spaces. Geometric
and Functional Analysis, Vol. 15, 5 (2005) 1114-1151.

\end{thebibliography}
\end{document}